\documentclass[11pt,reqno]{amsart}

\usepackage{amssymb,amsmath,amsthm,amscd,latexsym,amsfonts}
\usepackage{mathtools}
\usepackage[T1]{fontenc}
\usepackage{graphicx,epstopdf}
\usepackage{graphicx}
\usepackage{xcolor}
\usepackage{cite}
\usepackage{float}
\usepackage[a4paper,top=3cm, bottom=3cm, left=3.1cm, right=3.1cm]{geometry}

\newtheorem{thm}{Theorem}
\newtheorem{defn}{Definition}
\newtheorem{lemma}{Lemma}
\newtheorem{pro}{Proposition}
\newtheorem{rk}{Remark}

\newtheorem{ex}{Example}

\numberwithin{equation}{section} 

\newcommand{\bea}{\begin{eqnarray}}
\newcommand{\eea}{\end{eqnarray}}

\newcommand{\Z}{\mathbb{Z}}

\def\O{\Omega}

\def\O{\Omega}

\def\Z{\mathbb{Z}}

\begin{document}
\title [A non-linear second-order difference equation]
{A non-linear second-order difference equation related to   Gibbs measures of a SOS model}

\author {U.A. Rozikov}

\address{ U.A. Rozikov$^{a,b,c}$\begin{itemize}
 \item[$^a$] V.I.Romanovskiy Institute of Mathematics,  9, Universitet str., 100174, Tashkent, Uzbekistan;
\item[$^b$] AKFA University,
264, Milliy Bog street,	Yangiobod QFY, Barkamol MFY,
Kibray district, 111221, Tashkent region, Uzbekistan;
\item[$^c$] National University of Uzbekistan,  4, Universitet str., 100174, Tashkent, Uzbekistan.
\end{itemize}}
\email{rozikovu@yandex.ru}

\begin{abstract}
 For the SOS (solid-on-solid) model with an external field and with spin values from the set of all integers
 on a Cayley tree each (gradient) Gibbs measure corresponds to a boundary law (an infinite-dimensional vector function defined on vertices of the Cayley tree) satisfying a non-linear functional equation. Recently some translation-invariant and  height-periodic (non-normalisable) solutions to the equation are found. 
Here our aim is to find non-height-periodic and non-normalisable boundary laws for the SOS model. By such a solution one can construct a non-probability Gibbs measure. We find explicitly several non-normalisable boundary laws. Moreover, we reduce the problem to solving of a non-linear, second-order difference equation. We give analytic and numerical analysis of the difference equation. 

\end{abstract}
\maketitle

{\bf Mathematics Subject Classifications (2010).} 82B26 (primary);
60K35 (secondary)

{\bf{Key words.}} {\em SOS model, configuration, Cayley tree,
Gibbs measure, gradient Gibbs measures, boundary law}.

\section{Introduction}
In statistical physics for models with non-compact set of spin values the set of Gibbs measures may be empty. In such a situation some gradient or non-probability Gibbs measures (GMs) may exist (for detailed motivations and very recent results see \cite{BiKo}, \cite{BEvE}, \cite{EK}, \cite{FV}, \cite{HR}, \cite{HKLR}, \cite{HKa}, \cite{HKb}, \cite{Hphd}, \cite{KM}, \cite{KS}, \cite{LT}, \cite{Sh}, \cite{Z1}).   
It is known that for Hamiltonians given on a tree there is an one-to-one correspondence 
between (resp. gradient) Gibbs measures and (resp. non-normalisable, height-periodic) normalisable boundary laws
(see  \cite{HKLR},  \cite{HKa}, \cite{HKb}, \cite{KS}, \cite{Rjsp}, \cite{Z1}). These papers mainly devoted to models with a countable set of spin values and nearest-neighboring interactions. For such a model the boundary law is a tree-indexed family of infinite-dimensional vectors with positive coordinates. 
These vectors satisfy a non-linear equation, which is equivalent to a compatibility condition of finite-dimensional Gibbs distributions.   

By the non-linearity, infinite-dimensionality  and tree-indexness of the equation its analysis is very complicated. In above mentioned papers, for a class of Hamiltonians,  this equation mainly solved in class of translation-invariant (i.e. the vectors do not depend on vertices of the tree) and  height-periodic vectors (i.e. the coordinates of an infinite-dimensional vector is periodic). 

This paper is devoted to Gibbs measures of a SOS model with a countable set of spin values on Cayley trees.
Our aim is to find non-periodic and non-normalisable boundary laws for the SOS model and give corresponding to them non-probability Gibbs measures. First we find explicitly several non-normalisable boundary laws and then reduce the problem to solving of a non-linear, second-order difference equation. We give analytic and numerical analysis of the difference equation.

\subsection{Definitions and useful facts}
We consider SOS model, where spin-configuration $\omega$ is a function from the
vertices of the Cayley tree $\Gamma^k=(V, \vec L)$ (of order $k\geq 2$) to the set $\Z$ of integer numbers, where
$V$ is the set of vertices and $\vec L$ is the set of oriented edges (bonds) of the tree
(see \cite{Ro} and \cite{Robp} for properties of the Cayley tree and theory of Gibbs measures for models with finitely many spin values on trees).

For any configuration  $\omega = (\omega(x))_{x \in V} \in \mathbb Z^V$ and edge $b = \langle x,y \rangle$ of $\Gamma^k$
the \textit{difference} along the edge $b$ is given by $\nabla \omega_b = \omega(y) - \omega(x)$, where $\omega_b$ is a configuration on  $b = \langle x,y \rangle$, i.e., $\omega_b=\{\omega(x), \omega(y)\}$.  The configuration $\nabla \omega$ is called the \textit{gradient field} of $\omega$ (see \cite{KS}, \cite{K}).

The gradient spin variables are now defined by $\eta_{\langle x,y \rangle} = \omega(y) - \omega(x)$ for each $\langle x,y \rangle$.

The space of \textit{gradient configurations} is denoted by $\O^\nabla$. The measurable structure on the space $\Omega^{\nabla}$ is given by $\sigma$-algebra $$\mathcal{F}^\nabla:=\sigma(\{ \eta_b \, \vert \, b \in \vec L \}).$$

Let $\mathcal{T}_{\Lambda}^{\nabla}$ be the sigma-algebra of gradient configurations outside of the finite volume $\Lambda$ is generated by all gradient variables outside of $\Lambda$ and the relative height-difference on the boundary of $\Lambda$. 

For  $\omega:x\in V\mapsto \omega(x)\in \mathbb Z$, consider Hamiltonian of SOS model with external field  $\Phi:\mathbb Z\to \mathbb R$, i.e.,
\begin{equation}\label{nu1}
	H(\omega)=-J\sum_{\langle x,y\rangle: \atop x,y\in V}
	|\omega(x)-\omega(y)|+\sum_{x\in V}\Phi(\omega(x)),\end{equation}
where $J>0$.

For $b = \langle x,y \rangle$ and $\omega_b=\{\omega(x), \omega(y)\}$ define transfer operator 
\begin{equation}\label{Qd}
	Q_b(\omega_b) =\exp\left(- \beta J |\omega(x)-\omega(y)|+{1\over k+1}(\Phi(\omega(x))+ \Phi(\omega(y)))\right),
\end{equation}
and introduce the Markov (Gibbsian) specification as
$$
\gamma_\Lambda^\Phi(\sigma_\Lambda = \omega_\Lambda | \omega) = (Z_\Lambda^\Phi)(\omega)^{-1} \prod_{b \cap \Lambda \neq \emptyset} Q_b(\omega_b).
$$

If external field $\Phi(\cdot)\equiv 0$ then for any bond $b=\langle x,y \rangle$ the transfer operator $Q_b(\omega_b)$ is
a function of gradient spin variable $\zeta_b=\omega(y)-\omega(x)$, i.e., is 
a \textit{gradient interaction potential}.

\emph{Boundary law} (see \cite{Ge}, \cite{Z1}) which allow to describe the subset $\mathcal{G}(\gamma)$ of all Gibbs measures is defined as follows.

\begin{defn}\label{def} 
	
	\begin{itemize}
		\item	A family of vectors $\{ l_{xy} \}_{\langle x,y \rangle \in \vec L}$ with $l_{xy}=\left(l_{xy}(i) : i\in \Z\right) \in (0, \infty)^\Z$ is called a {\em boundary law for the transfer operators $\{ Q_b\}_{b \in \vec L}$} if for each $\langle x,y \rangle \in \vec L$ there exists a constant  $c_{xy}>0$ such that the consistency equation
		\begin{equation}\label{eq:bl}
			l_{xy}(i) = c_{xy} \prod_{z \in \partial x \setminus \{y \}} \sum_{j \in \Z} Q_{zx}(i,j) l_{zx}(j)
		\end{equation}
		holds for every $i \in \Z$, where $\partial x$ is the set of all nearest-neighbors of $x$. 
		\item  A boundary law $l$ is said to be {\em normalisable} if and only if
		\begin{equation}\label{Norm}
			\sum_{i \in \Z} \Big( \prod_{z \in \partial x} \sum_{j \in \Z} Q_{zx}(i,j) l_{zx}(j) \Big) < \infty
		\end{equation} at any $x \in V$.
		
		\item 	A boundary law 	is called {\em $q$-height-periodic} (or $q$-periodic) if $l_{xy} (i + q) = l_{xy}(i)$
		for every oriented edge $\langle x,y \rangle \in \vec L$ and each $i \in \Z$.
	\end{itemize}
\end{defn}

\begin{defn}  The gradient Gibbs specification is defined as the family of probability kernels $\left(\gamma_{\Lambda}^{\prime}\right)_{\Lambda \Subset V}$ from $\left(\Omega^{\nabla}, \mathcal{T}_{\Lambda}^{\nabla}\right)$ to $\left(\Omega^{\nabla}, \mathcal{F}^{\nabla}\right)$ such that
	$$
	\int F(\rho) \gamma_{\Lambda}^{\prime}(d \rho \mid \zeta)=\int F(\nabla \varphi) \gamma_{\Lambda}(d \varphi \mid \omega)
	$$
	for all bounded $\mathcal{F}^{\nabla}$-measurable functions $F$, where $\omega \in \Omega$ is any height-configuration with $\nabla \omega=\zeta$.
\end{defn}
\begin{defn} 
	A probability measure $\nu$ on $\Omega^{\nabla}$ is called a Gradient Gibbs Measure (GGM) if it satisfies the $DLR$ equation
	$$
	\int \nu(d \zeta) F(\zeta)=\int \nu(d \zeta) \int \gamma_{\Lambda}^{\prime}(d \tilde{\zeta} \mid \zeta) F(\tilde{\zeta})
	$$
	for every finite $\Lambda \subset V$ and for all bounded functions $F$ on $\Omega^{\nabla}$. 
\end{defn}

\begin{rk} Note that the problem of finding all boundary laws, i.e. all solutions of (\ref{eq:bl}) is by no means easy. Moreover, there is no theory of functional equations with the unknown	function defined on a tree. For a solution to (\ref{eq:bl}), i.e., a boundary law depending on its normalisablity and periodicity the following facts are known:

\begin{itemize}
	\item[F1.] There is an one-to-one correspondence between normalisable boundary laws 
	and Gibbs measures (also called tree-indexed Markov chains)   \cite{Z1}. 
	
	\item[F2.] Any $q$-height-periodic boundary law of SOS model (with zero external field) defines a translation invariant GGM (see \cite{KS} for the general formulation of this fact). 	
	
	As shown in (see \cite{HKLR} and \cite{HKb}) if a height-periodic boundary law is obtained from another one by a cyclic shift, then it leads to the same GGM.  Moreover, Theorem 5 in \cite{HKb}	guarantees that distinct (up to cyclic shift and multiplication by positive constants) boundary laws leads to distinct GGMs.
	
	More general theorem is given in \cite{HKa}, which says that 
	for any summable transfer operator $Q$ (in particular SOS model) and any degree $k\geq 2$ of the Cayley tree there is a finite period $q_0(k)$ such that for all $q\geq q_0(k)$ there are non-translation-invariant GGMs of period $q$.
		
Note that periodic boundary law is non-normalisable.  	In \cite{GRH}, \cite{HKLR},  \cite{HKa}, \cite{HKb}, \cite{KS}, \cite{Rjsp} some height-periodic  non-normalisable boundary laws are used to give GGMs. 

\item[F3.] Assume now we have a boundary law which is non-normalisable (and non-height-periodic in case of zero external field). What kind of measure can be constructed by such a boundary law? One possibility is to give a non-probability Gibbs measure (as already given in \cite{KM}). Namely, each {\it non-normalisable} boundary law $(l_{xy})_{x,y}$ for $(Q_b)_{b \in L}$ defines a non-probability Gibbs measure $\mu$ (having values in the extended real number line)\footnote{For general definition of measure see: https://en.wikipedia.org/wiki/Measure$_-$(mathematics)} via the equation given for any connected set $\Lambda \subset V$
\begin{equation}\label{BoundMC}
	\mu(\sigma_{\Lambda \cup \partial \Lambda}=\omega_{\Lambda \cup \partial \Lambda}) = \prod_{y \in \partial \Lambda} l_{y y_\Lambda}(\omega_y) \prod_{b \cap \Lambda \neq \emptyset} Q_b(\omega_b),
\end{equation}
where for any $y \in \partial \Lambda$, $y_\Lambda$ denotes the unique nearest neighbor of $y$ in $\Lambda$.

\end{itemize}
 \end{rk}

In this paper we study (gradient) Gibbs measures for SOS model by using the facts {\rm F1-F3}.

\subsection{The boundary law equation for the SOS model with an external field.}

In this paper we are interested to the translation-invariant boundary laws of the SOS model with external field, which are infinite-dimensional vectors of the form ${\bf z}=(z_i,\ i\in \mathbb Z)\in \mathbb R^\infty_+$, coordinates of it, after simplification (denoting $h(j)z_j$ by $z_j$) and normalization at $i=0$,  satisfy the following system of equations
\begin{equation}\label{di1}
	z_i=\frac{h(i)}{h(0)}\left({\theta^{|i|}+
		\sum_{j\in \mathbb Z_0}\theta^{|i-j|}z_j
		\over
		1+\sum_{j\in \mathbb Z_0}\theta^{|j|}z_j}\right)^k, \quad i\in\mathbb Z_0.
\end{equation}
Here $\mathbb Z_0=\mathbb Z\setminus\{0\}$, $h(i)=\exp(\Phi(i))$, $i\in \mathbb Z$,   $\theta=\exp(-J\beta)$. Note that $z_0=1$.

 
Let $\mathbf z(\theta)=(z_i=z_i(\theta)>0, i\in \mathbb Z_0)$ be a solution to (\ref{di1}).   Denote
\begin{equation}\label{lr}
	l_i\equiv l_i(\theta)=\sum_{j=-\infty}^{-1}\theta^{|i-j|}z_j, \quad
	r_i\equiv r_i(\theta)=\sum_{j=1}^{\infty}\theta^{|i-j|}z_j, \quad i\in\mathbb Z_0.
\end{equation}
It is clear that each $l_i$ and $r_i$ can be a finite positive number or $+\infty$.

\begin{lemma}\label{l1} \cite{HKLR} For each $i\in \mathbb Z_0$ we have
	\begin{itemize}
		\item $l_i<+\infty$ if and only if $l_0<+\infty$;
		
		\item $r_i<+\infty$ if and only if $r_0<+\infty$.
	\end{itemize}
\end{lemma}
\section{Non-zero external field: non-probability Gibbs measures}

In \cite{HR} for the case of 4-periodic non-zero external field and  $k=2$ some 4-periodic boundary laws are found. In this  section we assume $\sum_{i\in \mathbb Z}h(i)=1$, i.e., $h$ is a probability measure and shall describe solutions of (\ref{di1}) with property  $l_0=+\infty$, or $r_0=+\infty$.

\subsection{The set of solutions with $l_0=+\infty$, $r_0<+\infty$}

In this case, using Lemma \ref{l1} from (\ref{di1}) we get
\begin{equation}\label{nu12}
	z_i=\frac{h(i)}{h(0)} \theta^{ik}, \quad \ i\in\mathbb Z_0.
\end{equation}
We should have
\begin{equation}\label{lin}
	+\infty=l_0=\sum_{j=-\infty}^{-1}\theta^{-j}z_j={1\over h(0)}\sum_{j=1}^\infty \theta^{-j(k-1)}h(-j),
\end{equation}
\begin{equation}\label{rfi}
	+\infty>r_0=\sum_{j=1}^\infty\theta^{j}z_j={1\over h(0)}\sum_{j=1}^\infty \theta^{j(k+1)}h(j).
\end{equation}
Thus if the measure $h$ and $\theta$ satisfy conditions (\ref{lin}) and (\ref{rfi}) then the system (\ref{di1}) has a solution (\ref{nu12}).
\begin{rk} Here we shall give some examples:
	\begin{itemize}
		\item Case: $\theta<1$. It is easy to see that if $\theta<1$ then the condition (\ref{rfi}) is satisfied for any (probability) measure $h$. One can (for $\theta<1$) choose a probability measure $h$ which  satisfies (\ref{lin}). For example, take $h$ as
		\begin{equation}
			\label{nuj}
			h(j)={1-\theta\over 1+\theta}\theta^{|j|},\quad j\in \mathbb Z
		\end{equation}
		then the condition (\ref{lin}) is satisfied for any $k\geq 2$. For this example the solution (\ref{nu12}) is
		\begin{equation}\label{s1}
			z_i=\theta^{ik+|i|}, \quad i\in\mathbb Z_0.
		\end{equation}
		\item
		Case: $\theta>1$. One can choose a measure $h$ such that (\ref{rfi}) is satisfied (take, for example, $h(j)={\theta^{k+2}-1\over \theta^{k+2}+1}\theta^{-|j|(k+2)}$). But the condition (\ref{lin}) is never satisfied for $\theta>1$.
	\end{itemize}
\end{rk}
\subsection{The set of solutions with  $l_0<+\infty$, $r_0=+\infty$}
We have
\begin{equation}\label{nu12b}
	z_i=\frac{h(i)}{h(0)} \theta^{-ik}, \quad \ i\in\mathbb Z_0.
\end{equation}
We must have
\begin{equation}\label{lfi}
	+\infty>l_0={1\over h(0)}\sum_{j=1}^\infty \theta^{j(k+1)}h(-j),
\end{equation}
\begin{equation}\label{rin}
	+\infty=r_0={1\over h(0)}\sum_{j=1}^\infty \theta^{-j(k-1)}h(j).
\end{equation}
Hence if the measure $h$ and $\theta$ satisfy conditions (\ref{lfi}) and (\ref{rin}) then the system (\ref{di1}) has a solution (\ref{nu12b}).

Here we shall give some examples:
\begin{ex}
	\begin{itemize}
		\item Case: $\theta<1$. It is easy to see that if $\theta<1$ then the condition (\ref{lfi}) is satisfied for any $h$. One can choose a probability measure $h$ which  satisfies (\ref{rin}). For example, if $h$ as (\ref{nuj}) then the condition (\ref{rin}) is satisfied for any $k\geq 2$. For this example the solution (\ref{nu12b}) is
		\begin{equation}\label{s2}
			z_i=\theta^{-ik+|i|}, \quad i\in\mathbb Z_0.
		\end{equation}
		\item
		Case: $\theta>1$. One can choose a measure $h$ such that (\ref{lfi}) is satisfied (take, for example, $h(j)={\theta^{k+2}-1\over \theta^{k+2}+1}\theta^{-|j|(k+2)}$). But the condition (\ref{rin}) is never satisfied for $\theta>1$.
		\item From above-mentioned results it follows that if $\theta<1$ then for the {\it fixed} measure  $h$ given by (\ref{nuj}) we have {\it two} solutions
		(\ref{s1}) and (\ref{s2}).
	\end{itemize}
\end{ex}
\subsection{The set of solutions with  $l_0=+\infty$, $r_0=+\infty$}
In this case we denote $\rho=r_0/l_0$, it is clear that $0\leq\rho\leq+\infty$.

From (\ref{di1}) we get
\begin{equation}\label{nu12c}
	z_i=\frac{h(i)}{h(0)} \left({\theta^i+\theta^{-i}\rho\over 1+\rho}\right)^k, \quad \ i\in\mathbb Z_0.
\end{equation}

{\it Case: $\rho=0$}. For $\rho=0$ we must have
\begin{equation}\label{lin0}
	+\infty=l_0={1\over h(0)}\sum_{j=1}^\infty \theta^{-j(k-1)}h(-j),
\end{equation}
\begin{equation}\label{rin0}
	+\infty=r_0={1\over h(0)}\sum_{j=1}^\infty \theta^{j(k+1)}h(j).
\end{equation}
It is easy to see that the conditions (\ref{lin0}) and (\ref{rin0}) can
not be satisfied simultaneously. Indeed, for $\theta<1$ (resp. $\theta>1$) one has
$r_0$ (resp. $l_0$) is finite for any probability measure $h$. Thus there is no any solution with $\rho=0$.

{\it Case: $\rho=\infty$.}  This case is similar to the case $\rho=0$ and there is no any solution with $\rho=\infty$.

{\it Case: $0<\rho<\infty$.} In this case for $z_i$ given by (\ref{nu12c}) we should have
\begin{equation}\label{lin1}
	+\infty=l_0={1\over h(0)(1+\rho)^k}\sum_{j=1}^\infty \theta^j(\theta^{-j}+\theta^j\rho)^kh(-j),
\end{equation}
\begin{equation}\label{rin1}
	+\infty=r_0={1\over h(0)(1+\rho)^k}\sum_{j=1}^\infty \theta^j(\theta^j+\theta^{-j}\rho)^kh(j),
\end{equation}
\begin{equation}\label{tau}
	{\sum_{j=1}^\infty \theta^j(\theta^j+\theta^{-j}\rho)^kh(j)\over
		\sum_{j=1}^\infty \theta^j(\theta^{-j}+\theta^j\rho)^kh(-j)}=\rho.
\end{equation}
Using the binomial formula we rewrite condition (\ref{lin1}) in the following form
$$+\infty=\sum_{j=1}^\infty \theta^j(\theta^{-j}+\theta^j\rho)^kh(-j)=\sum_{s=0}^k{k\choose s }\rho^s\sum_{j=1}^\infty \theta^{-j(k-2s-1)}h(-j).$$
This condition is equivalent to
\begin{equation}
	\label{scl}
	\sum_{j=1}^\infty \theta^{-j(k-2s-1)}h(-j)=+\infty, \quad \mbox{for some} \quad s\in \{0,1,\dots,k\}.
\end{equation}
Similarly, we have that the condition (\ref{rin1}) is equivalent to the following
\begin{equation}
	\label{scr}
	\sum_{j=1}^\infty \theta^{j(k-2t+1)}h(j)=+\infty, \quad \mbox{for some} \quad t\in \{0,1,\dots,k\}.
\end{equation}
Thus if conditions (\ref{tau})-(\ref{scr}) are satisfied then the system (\ref{di1}) has a solution of the form (\ref{nu12c}).

\begin{rk}\label{R} If $h$ is symmetric, i.e. $h(j)=h(-j)$ then the condition (\ref{tau}) is satisfied with $\rho=1$ for any $\theta$. Moreover, if $\rho$ is a solution to (\ref{tau}) (for the symmetric measure) then $1/\rho$ is also a solution to (\ref{tau}).
\end{rk}

\begin{ex} Consider the case $\theta<1$, $k\geq 2$ and the measure (\ref{nuj}) then it is easy to see that the conditions (\ref{scl}) and (\ref{scr}) are satisfied. Indeed, in (\ref{scl}) (resp. (\ref{scr})) we take any $s$ (resp. $t$) with $s\leq {k-2\over 2}$ (resp. $t\geq {k+2\over 2}$). Since the measure (\ref{nuj}) is symmetric by Remark \ref{R} we know that the condition (\ref{tau}) is also satisfied.  Thus for measure (\ref{nuj}) and $\theta<1$ we have solution
	\begin{equation}\label{s3}
		z_i=\theta^{|i|} \left({\theta^i+\theta^{-i}\over 2}\right)^k, \quad \ i\in\mathbb Z_0.
	\end{equation}
\end{ex}

As mentioned above the solutions given in this section do not define a (gradient) Gibbs measure, since they are not normalisable and non-height-periodic. But by {\rm F3} the above mentioned solutions define non-probability Gibbs measures. We summarize above obtained examples of solutions in the following.
\begin{thm} If parameters of SOS model (\ref{nu1}) are such that 
	$$ \theta=\exp(-J\beta)<1, \ \   h(i)=\exp(\Phi(i))={1-\theta\over 1+\theta}\theta^{|i|},   i\in \mathbb Z$$
	then there are at least three non-probability Gibbs measures corresponding to explicit solutions (\ref{s1}), (\ref{s2}), (\ref{s3}).
\end{thm} 	  
More general conditions on $\theta$ and probability measure $h$ are given in (\ref{lin}), (\ref{rfi}), (\ref{lfi}), (\ref{rin}), and (\ref{tau})-(\ref{scr}). Taking an arbitrary $\theta$ and $h$ one can check these conditions and obtain several solutions (depending on values of parameter $\rho$) and corresponding non-probability Gibbs measures. 

\section{Difference equation corresponding to system (\ref{di1}).}
\begin{pro}\cite{HR} \label{pps: 1} Assume $h(0)=1$.
	A vector $\mathbf z=(z_i,i\in \mathbb Z)$, with $z_0=1$,  is a solution to (\ref{di1})
	if and only if for $u_i=\sqrt[k]{{z_i\over h(i)}}$ the following holds
	\begin{equation}\label{Va}
		h(i)u_i^k=\frac{{u_{i-1}+u_{i+1}-\tau u_i}}{u_{-1}+u_{1}-\tau}, \quad i\in \mathbb Z,
	\end{equation}
	where $\tau=\theta^{-1}+\theta=2\cosh(\beta)$.
\end{pro}

Rewrite (\ref{Va}) as (forward)
\begin{equation}\label{uf}
u_{i+1}=(u_{-1}+u_{1}-\tau)h(i)u_i^k
+\tau u_i- u_{i-1}, \quad i\in \mathbb Z,
\end{equation}
or (backward) form
\begin{equation}\label{ub}
	u_{i-1}=(u_{-1}+u_{1}-\tau)h(i)u_i^k
	+\tau u_i- u_{i+1}, \quad i\in \mathbb Z.
\end{equation}
Therefore, it suffices to study (\ref{uf}) for $i\geq -1$.
Moreover, following \cite{HR} and \cite{HKLR} we are interested to $u_i>0$, $i\geq -1$ which satisfies
\begin{equation}\label{uu}
	u_{i+1}=(u_{-1}+u_{1}-\tau)h(i)u_i^k
	+\tau u_i- u_{i-1}, \quad i=0,1,2,\dots,
\end{equation}
with initial conditions:
\begin{equation}\label{ui}
	h(0)=u_0=1, \ \ u_{-1}+u_1<\tau.
\end{equation}	

Note that for $i=0$ the above equation is trivially fulfilled for all values of $u_1$ and $u_{-1}$.

{\bf The main problem} is to find $u_1$ and $u_{-1}$ such that the sequence  $\{u_i\}_{i=-1}^\infty$ generated by (\ref{uu}) and (\ref{ui}) is strictly positive. 

\begin{pro}\label{pb} Let $k\geq 2$. If a strictly positive sequence $\{u_i\}_{i=-1}^\infty$ 
	satisfies (\ref{uu}), (\ref{ui}) then it is bounded.
\end{pro}
\begin{proof} This is generalization of Proposition 2 in  \cite{Rjsp}. The main reason for boundedness is that the coefficient $(u_{-1}+u_{1}-\tau)h(i)$ of $u_i^k$ is negative (by the condition on initial point) therefore, to have $u_{i+1}>0$
the value $u_i>0$ should be bounded.
\end{proof}

The equation (\ref{uu}) is a non-linear (for $k\geq 2$), second-order difference equation (see for example \cite{El}).

Some $q$-periodic and mirror symmetric solutions of (\ref{uu}) are found in recent paper \cite{Rjsp}. Here we are interested to non-periodic solutions.

Let us reduce the problem to a dynamical system. Denoting $x_{n}=u_n$ and $y_n=u_{n-1}$ the difference equation (\ref{uu}) can be reduced to the discrete-time dynamical system
\begin{equation}\label{ud} (x_{n+1}, y_{n+1})=F_n(x_n,y_n), \ \ n=0,1,2,\dots
\end{equation}
where $y_0+x_1<\tau$, $x_0=1$ and the operator $F_n:(x,y)\in \mathbb R^2\to F_n(x,y)=(x',y')\in \mathbb R^2$ is defined as
\begin{equation}\label{uo}
F_n: \ \	\begin{array}{ll}
	x'=(y_{0}+x_{1}-\tau)h(n)x^k+\tau x-y\\[2mm]
	y'=x
\end{array}
\end{equation}
\begin{rk}
	\begin{itemize}
		\item[1.] The dynamical system (\ref{ud}) is complicated, because it is non-linear (for any $k\geq 2$), the operator $F_n$ depends on initial point $(y_0,x_1)$ and depends on the time $n$.
		\item[2.] In case when $h(0)=1$, $h(n)$ is constant (i.e., independent on $n\geq 1$) then
		 $F_0$ is
		 \begin{equation}\label{u0}
			F_0: \ \	\begin{array}{ll}
				x'=(y_{0}+x_{1}-\tau)x^k+\tau x-y\\[2mm]
				y'=x.
			\end{array}
		\end{equation}
		and	 $F_n$ also does not depend on $n\geq 1$, i.e., $F_n=F$ with
		 \begin{equation}\label{un}
		 	F: \ \	\begin{array}{ll}
		 		x'=(y_{0}+x_{1}-\tau)hx^k+\tau x-y\\[2mm]
		 		y'=x.
		 	\end{array}
		 \end{equation}
Moreover,	$$F_0(x_0,y_0)=F_0(1,y_0)=(x_1,1).$$
For given $x_1>0$ we define the trajectory of point $(x_1,1)$ as
\begin{equation}\label{tx}
	(x_m,y_m)=F^m(x_1,1), \ \ m\geq 1.
\end{equation}
	 For $k=2$ the operator (\ref{un})
	 is similar to the H\'enon map, $H:\mathbb R^2\to \mathbb R^2$, defined by
	 $$H: \left\{\begin{array}{ll}
	 	x'=1+y-ax^2\\[2mm]
	 	y'=bx, \ \ \ \ a,b>0.
	 \end{array}\right.$$
	 which is quadratic map in dimension
	 two \cite[page 251]{De}. It is known that for some values of its parameters, the dynamics of
	 the H\'enon map is very complex, having infinitely many periodic points. This is one of the
	 most studied examples of dynamical systems that exhibit chaotic behavior.
	  \item[3.] For $k\geq 3$ the operator (\ref{un}) is known as a McMillan map (see \cite{FM}, \cite{JR}, \cite{JRv} and references therein). Below we will study dynamical system generated by (\ref{un}).
	\end{itemize}
	
\end{rk}

\subsection{Fixed points of (\ref{un})}
Fixed points are solutions to the following system of equations
 \begin{equation}\label{unf}
	\begin{array}{ll}
		x=(y_{0}+x_{1}-\tau)hx^k+\tau x-y\\[2mm]
		y=x.
	\end{array}
\end{equation}
It is easy to see that this equation has two solutions:
${\rm Fix}(F)=\{P_0, P_1\}$, with
\begin{equation}\label{fix}
 P_0=(0,0), P_1=(x^*, x^*), \ \ \mbox{where} \ \ x^*=\left({\tau-2\over h(\tau-y_0-x_1)}\right)^{{1\over k-1}}.
 \end{equation}
Now we shall examine the type of the fixed points.

\begin{defn}\label{d2} (see \cite{De}). A fixed point $v$ of an operator $M$ is called hyperbolic if
	its Jacobian $J$ at $v$ has no eigenvalues on the unit circle.
	
	A hyperbolic fixed point $v$ is called:
	\begin{itemize}
		\item attracting if all the eigenvalues of the Jacobi matrix $J(v)$ are less than 1 in
		absolute value;
		\item repelling if all the eigenvalues of the Jacobi matrix $J(v)$ are greater than 1 in
		absolute value;
		\item a saddle point otherwise.
	\end{itemize}
\end{defn}

To find the type of a fixed point of the operator (\ref{un}) we write
the Jacobi matrix:
$$J(x,y)=\left(\begin{array}{cc}
(y_{0}+x_{1}-\tau)hkx^{k-1}+\tau & -1\\[2mm]
1& 0
\end{array}	
\right).$$
Note that both fixed points are saddle (or non-hyperbolic), because the equation for eigenvalues is
$$\lambda^2-((y_{0}+x_{1}-\tau)hkx^{k-1}+\tau)\lambda+1=0$$
solutions of which satisfy $\lambda_1\lambda_2=1$.
The eigenvalues for $P_0$ are:
$$\lambda_{1,2}={1\over 2}\left(\tau\mp \sqrt{\tau^2-4}\right),$$
and for $P_1$ are:
$$\lambda_{1,2}={1\over 2}\left(2k-(k-1)\tau\mp 
\sqrt{(k-1)(\tau-2)[(k-1)\tau-2(k+1)]}\right).$$

Therefore we obtain the following proposition

\begin{pro}\label{pfp} For  $k\geq 2$ and  $\tau>2$ the following assertions hold
	\begin{itemize} 
		\item [1.] Eigenvalues corresponding to the fixed point $P_0$ are real and  $0<\lambda_1<1$, $\lambda_2>1$.
	\item[2.]	For the eigenvalues corresponding to $P_1$ the following hold
	\begin{itemize}
		\item[2.1.] If  $2<\tau<{2(k+1)\over k-1}$ then the eigenvalues are complex numbers, and (surprisingly) 	$|\lambda_{1,2}|=1$ independently on values of the parameters $\tau$ and $k$. 
		\item[2.2.]   If  $\tau={2(k+1)\over k-1}$ then 
		$\lambda_1=\lambda_2=-1$.
		\item[2.3.]  If  $\tau>{2(k+1)\over k-1}$ then both eigenvalues are real and $|\lambda_1|<1$, $|\lambda_2|>1$. 
	\end{itemize}
\end{itemize}
\end{pro}
\begin{proof}
	The proof follows from the formulas of $\lambda_{1,2}$ given above. Let us give proof of the item 2.1.  Under condition of this item the eigenvalues are complex:
	$$\lambda_{1,2}={1\over 2}\left(2k-(k-1)\tau\mp 
	i\sqrt{(k-1)(\tau-2)[2(k+1)-(k-1)\tau]}\right).$$
Therefore, simple computations show that
$$|\lambda_{1,2}|^2={1\over 4}\left((2k-(k-1)\tau)^2+\left((k-1)(\tau-2)[2(k+1)-(k-1)\tau]\right)^2\right)=1,$$
independently on parameters $\tau$ and $k$ satisfying the condition of this item.
\end{proof}
Thus $P_0$ is always saddle. But $P_1$ is saddle if $\tau>{2(k+1)\over k-1}$ and non-hyperbolic if $2<\tau\leq {2(k+1)\over k-1}$. 

   From the known theorem about stable and unstable manifolds (see \cite{De})
  we get the following result
  \begin{thm} For the saddle fixed points $P_i$ the following assertions hold
  	 \begin{itemize} 
  		\item[a.] There are curves
  		denoted by $W^{s}(P_0)$, $W^{s}(P_1)$ such that for any initial
  		vector $v\in {\displaystyle W^{s}(P_i)}$ one has $F(v)\in {\displaystyle W^{s}(P_i)}$ (invariance) and
  		$$\lim_{n\to\infty} F^n(v)=P_i, \, i=0,1.$$
  		\item[b.] There are curves denoted by  $W^{u}(P_0)$
  		(resp. $W^{u}(P_1)$) and a neighborhood $V(P_i)$ of $P_i$ such that for any initial vector $v\in W^{u}(P_i)\cap V(P_i)$, there exists $k=k(v)\in \mathbb N$ that $F^k(v)\notin V(P_i)$, $i=0,1$.
  	\end{itemize}
  \end{thm}
  The curves $W^{s}(P_i)$ are known as stable and $W^{u}(P_i)$ as unstable manifolds  (see \cite{Ga}).\\

 {\bf Open problem:} Find explicit formula of the curves $W^{s}(P_i)$, $W^{u}(P_i)$.

\subsection{An invariant set} For given parameters and fixed $x_1$, $y_0$ we introduce
$$\psi(x)=(y_{0}+x_{1}-\tau)hx^k+\tau x,$$
$$a=\left({\tau-1\over h(\tau-y_0-x_1)}\right)^{1/(k-1)}.$$
\begin{pro} If $2<\tau\leq 1+{k^{k/(k-1)}\over k-1}$ then the set
$$I=\{(x,y)\in \mathbb R^2_+: 0\leq x\leq a, \ \ \max\{0, \psi(x)-a\}\leq y\leq \psi(x)\},$$ is invariant with respect to operator $F$ given in (\ref{un}), i.e., $F(I)\subset I$.
\end{pro}
\begin{proof} We take an arbitrary $(x,y)\in I$ and show that $(x',y')=F(x,y)\in I$.  Since $y_0+x_1<\tau$ the function $\psi(x)$ is monotone increasing in $(0, \hat x)$, (where $\hat x=\hat x_0\sqrt[k-1]{1/k}$,  $\hat x_0=\sqrt[k-1]{\tau\over h(\tau-y_0-x_1)}$) and decreasing in $ (\hat x, +\infty)$. Note that $\hat x_0>\hat x$, $\psi(x)>0$ for all $x\in (0, \hat x_0)$,  $\psi(\hat x_0)=0$  and $\psi(x)<0$ for all $x>\hat x_0$.
	
	We have $a<\hat x_0$ and therefore $\psi(x)>0$ for all $x\in (0, a)$. Using these inequalities we obtain
	$$0\leq x'\leq a \ \ \Leftrightarrow \ \ 0\leq \psi(x)-y\leq a  \ \ \Leftrightarrow \ \ (\mbox{since} \ \ y\geq 0) \ \  \max\{0, \psi(x)-a\}\leq y\leq \psi(x). $$
	For the second coordinate we have
$$\max\{0, \psi(x)-a\}\leq y'\leq \psi(x)  \ \ \Leftrightarrow \ \ \max\{0, \psi(x)-a\}\leq x\leq \psi(x) \ \ \Leftrightarrow \ \ $$
$$ \left\{\begin{array}{lll}
	\psi(x)-x\geq 0 \\[2mm]
	 \psi(x)-a-x\leq 0\\[2mm]
	 \psi(x)-a\geq 0
	 \end{array}\right. \ \ \mbox{or} \ \ \left\{\begin{array}{lll}
	 \psi(x)-x\geq 0 \\[2mm]
	 x\geq 0\\[2mm]
	 \psi(x)-a\leq 0
 \end{array}\right.
 $$
 Note that $\psi(x)-x\geq 0$ iff $x\in [0, a]$. Moreover,
 if $2<\tau\leq 1+{k^{k/(k-1)}\over k-1}$ then
  $$\max_{x\in [0, a]}(\psi(x)-x)=\psi(x_*)-x_*=(\tau-1){k-1\over k}x_*<a,$$
  where
  $$x_*=\left({\tau-1\over hk(\tau-y_0-x_1)}\right)^{1/(k-1)}.$$
  Hence, $ \psi(x)-a-x\leq 0$, and the proof is completed.
\end{proof}
\subsection{Inverse of the operator (\ref{un})}		
Recall the following definition. Let
$f\colon X\to X$, $g\colon Y\to Y$, and $\ell\colon Y\to X$ are continuous functions on topological spaces, $X$ and $Y$.

The mapping $f$ is called topologically semiconjugate to $g$ if $\ell$ is a surjection such that $f\circ \ell=\ell \circ g$.

$f$ and $g$ are topologically conjugate if they are topologically semiconjugate and $\ell$ is a homeomorphism.
\begin{pro}
	Operator (\ref{un}) is invertible and its inverse
is	\begin{equation}\label{uni}
		F^{-1}: \ \	\begin{array}{ll}
			x'=y\\[2mm]
			y'=(y_{0}+x_{1}-\tau)hy^k+\tau y-x.
		\end{array}
	\end{equation}
Moreover, the operator $F$ is topological conjugate to its inverse.
\end{pro}
\begin{proof}
	To find inverse of the operator (\ref{un}) from the second equation of the operator one can find $x=y'$ and then putting it in the first equation one gets
	$$y= (y_{0}+x_{1}-\tau)h(y')^k+\tau y'-x'.$$
Thus we get $F^{-1}$. To see that $F$ and $F^{-1} $	are conjugate
take mapping $\ell: \mathbb R^2\to \mathbb R^2$ defined by
$\ell(x,y)=(y,x)$, i.e., permutation of the coordinates. Note that $\ell^{-1}=\ell$.
Then it is easy to see that
$$F(x,y)=\ell\circ F^{-1} \circ \ell (x,y).$$
\end{proof}
\section{Bifurcations}
It is known (see \cite{Ga}, \cite{Kub}) that if the repeating eigenvalue is in absolute value equal to 1 (as in our case $\lambda_1=\lambda_2=-1$) then the system is characterized by a continuum of unstable equilibrium. This non-generic case represents the bifurcation point of the dynamical system. Meaning that an infinitesimal change in the value of the eigenvalue brings about a qualitative change in the nature of the dynamical system.

In this section following \cite{Ku}, \cite{Kub}  we study bifurcations of the dynamical system (\ref{un}).

\subsection{1:2 resonance} In the case $\lambda_1=\lambda_2=-1$ the dynamical system is characterized as $1: 2$ resonance (see page 415 of \cite{Kub})). This case appears for the fixed point $P_1$ (see Proposition \ref{pfp}). Changing $x$ and $y$ by $x-x^*$ and $y-x^*$ (see (\ref{fix})) respectively, we can use Lemma 9.8 of \cite{Kub}, which states that the normal form map for $1: 2$ resonance is the map (denoted by $\Gamma_{\beta}(\xi)$)
$$
\Gamma_{\beta} : \left(\begin{array}{l}
	\xi_{1} \\
	\xi_{2}
\end{array}\right) \mapsto\left(\begin{array}{cc}
	-1 & 1 \\
	\beta_{1} & -1+\beta_{2}
\end{array}\right)\left(\begin{array}{c}
	\xi_{1} \\
	\xi_{2}
\end{array}\right)+\left(\begin{array}{c}
	0 \\
	C(\beta) \xi_{1}^{3}+D(\beta) \xi_{1}^{2} \xi_{2}
\end{array}\right)+O\left(\|\xi\|^{4}\right),
$$
where $\beta_{1}$ and $\beta_2$ are parameters and   $C(\beta)$ and $D(\beta)$ are smooth functions.

 Note that the linear part of $\Gamma_\beta$, for $\beta=0$,
has negative eigenvalues, therefore it can not by approximated by a flow. But the second iterate, $\xi \mapsto \Gamma_{\beta}^{2}(\xi),$
can be approximated by the unit-time shift of a flow. The map $\Gamma_{\beta}^{2}$ has the form (see \cite{Kub}):
$$
\left(\begin{array}{l}
	\xi_{1} \\
	\xi_{2}
\end{array}\right) \mapsto\left(\begin{array}{cc}
	1+\beta_{1} & -2+\beta_{2} \\
	-2 \beta_{1}+\beta_{1} \beta_{2} & 1+\beta_{1}-2 \beta_{2}+\beta_{2}^{2}
\end{array}\right)\left(\begin{array}{l}
	\xi_{1} \\
	\xi_{2}
\end{array}\right)+\left(\begin{array}{c}
	V(\xi, \beta) \\
	W(\xi, \beta)
\end{array}\right),
$$
where $	V(\xi, \beta)$, $W(\xi, \beta)$ are cubic polynomials of two-variables $\xi_1$, $\xi_2$.
otherwise, reverse time.

In page 424 of \cite{Kub}, under some assumptions, the system is approximated by
$$
\left\{\begin{array}{l}
	\dot{\zeta}_{1}=\zeta_{2} \\
	\dot{\zeta}_{2}=\varepsilon_{1} \zeta_{1}+\varepsilon_{2} \zeta_{2}+s \zeta_{1}^{3}-\zeta_{1}^{2} \zeta_{2},
\end{array}\right.
$$
where $s=\pm 1$.
Moreover, in \cite{Kub} the bifurcation diagrams of this system for $s=1$ and $s=-1$ are presented. These general results are true for our system (\ref{un}) for $\lambda_1=\lambda_2=-1$.
In Figure \ref{fig-11} a trajectory is given.
\begin{figure}[h!]
	\includegraphics[width=8cm]{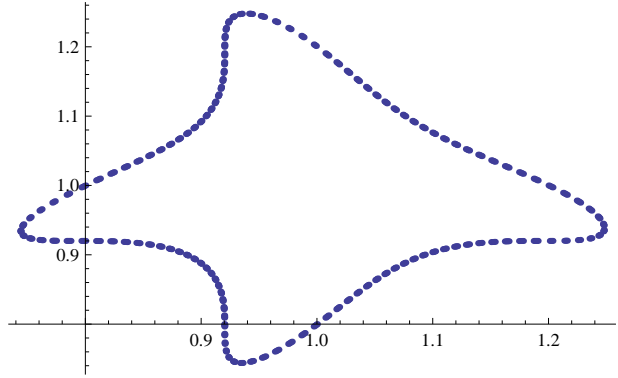}
	\caption{ The trajectory for the case $k=3$, $h=1$, $\tau=4$, $x_0=1$, $y_0=1.2$, $x_1=0.8$, the number of steps is $n=500$. In this case $\lambda_1=\lambda_2=-1.$}\label{fig-11}
\end{figure}

\subsection{Neimark-Sacker bifurcation} Let $2<\tau<{2(k+1)\over k-1}$.   In this case the fixed point $P_1$ has simple critical eigenvalues $\lambda_{1,2}=e^{\pm i \theta_{\tau}}$, where
$$\theta_\tau:= \arctan\left({2k-(k-1)\tau\over
	\sqrt{(k-1)(\tau-2)(2(k+1)-(k-1)\tau)}}\right).$$

We note that 
$$
e^{i q \theta_{\tau}}-1 = 0 \ \ \Leftrightarrow \ \ \tau={2k\over k-1}.
$$
This case is known as a strong resonances \cite{Kub}.

Assume $\tau\in \left(2, {2(k+1)\over k-1} \right)\setminus\{{2k\over k-1}\}$, then by results of Chapters 3-4 in \cite{Kub} there is a unique closed invariant curve around the fixed point when the parameter $\tau$ crosses the critical value. Below we give several pictures of trajectories of the dynamical system.

\section{Some trajectories of the operator (\ref{un})}	
In general, it is very difficult to study limit behavior of each trajectory of the operator (\ref{un}). 

\subsection{Case $h=1$} For $h=1$, i.e., zero external field, the following numerical results (see Fig. \ref{fig-1}-\ref{fig-6}) show that  trajectories, defined in (\ref{tx}), do not have any limit point. Moreover, each trajectory consists a dense subset of a closed curve.    
\begin{figure}[H]
	\includegraphics[width=8cm]{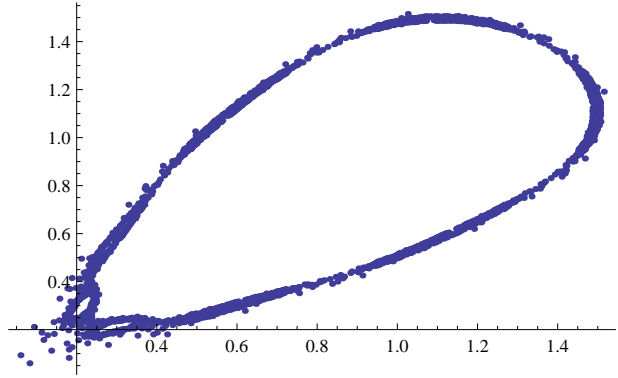}
	\caption{The trajectory for the case $k=2$, $h=1$, $\tau=3$, $x_0=1$, $y_0=0.5$,  $x_1=1.48589$, shown $n=3000$ iterations.}\label{fig-1}
\end{figure}

\begin{figure}[H]
	\includegraphics[width=8cm]{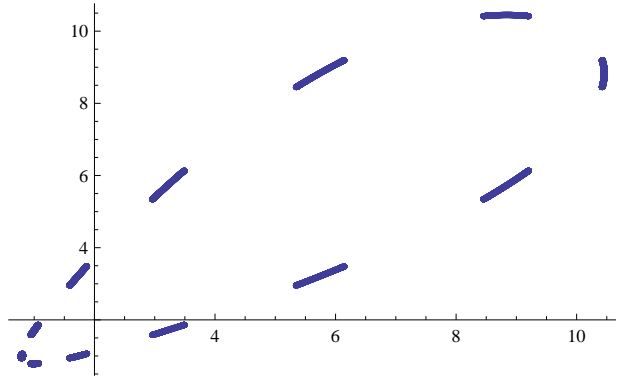}
	\caption{ The trajectory for the case $k=2$, $h=1$, $\tau=2.6$, $x_0=1$, $y_0=0.8$, $x_1=1.713$, $n=10000$.}\label{fig-2}
\end{figure}

\begin{figure}[H]
	\includegraphics[width=8cm]{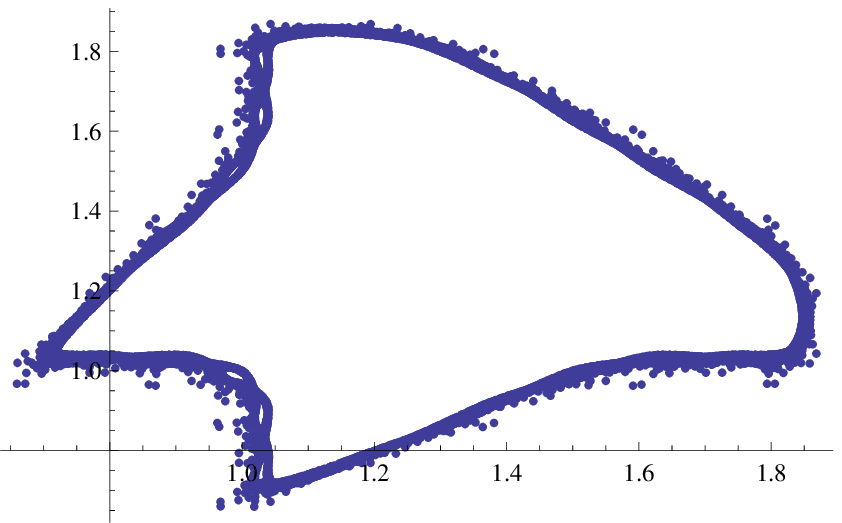}
	\caption{ The trajectory for the case $k=2$, $h=1$, $\tau=4$, $x_0=1$, $y_0=1.5$, $x_1=1$, $n=10000$.}\label{fig-3}
\end{figure}

\begin{figure}[H]
	\includegraphics[width=8cm]{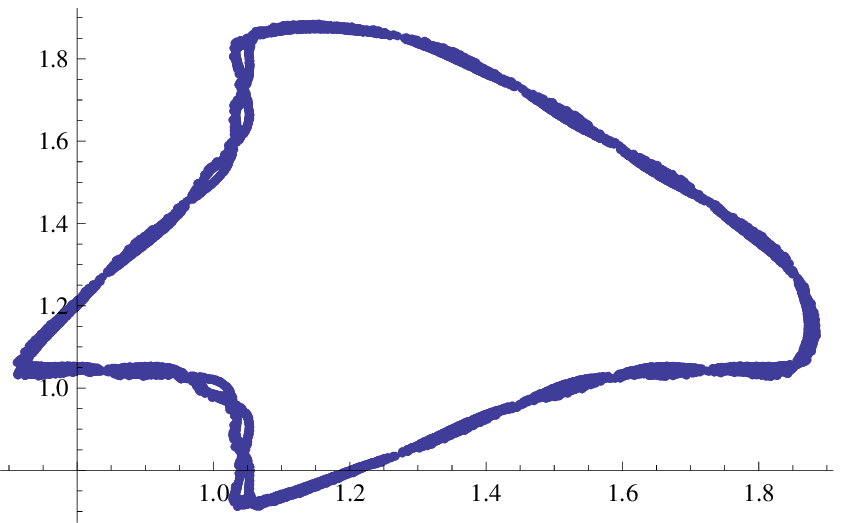}
	\caption{ The trajectory for the case $k=2$, $h=1$, $\tau=4$, $x_0=1$, $y_0=1.5$, $x_1=1.02$, $n=10000$.}\label{fig-4}
\end{figure}

\begin{figure}[H]
	\includegraphics[width=8cm]{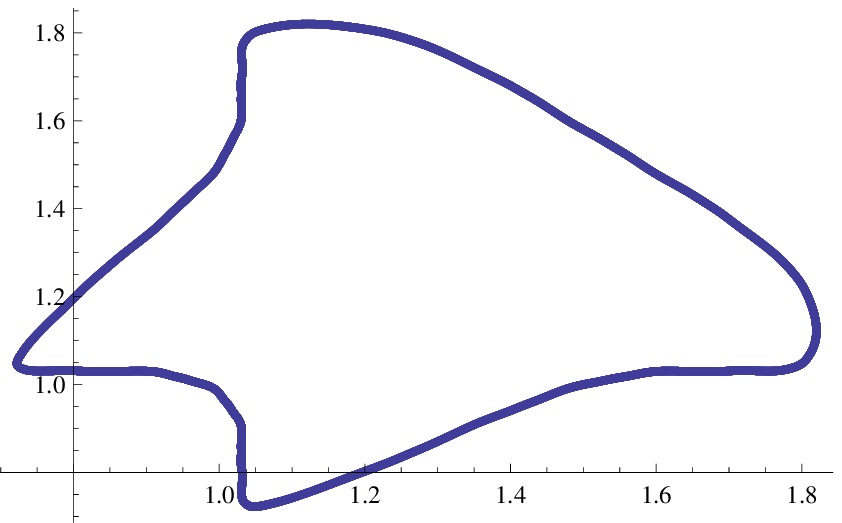}
	\caption{ The trajectory for the case $k=2$, $h=1$, $\tau=4$, $x_0=1$, $y_0=1.5$, $x_1=0.98$, $n=10000$.}\label{fig-5}
\end{figure}
\begin{figure}[H]
	\includegraphics[width=8cm]{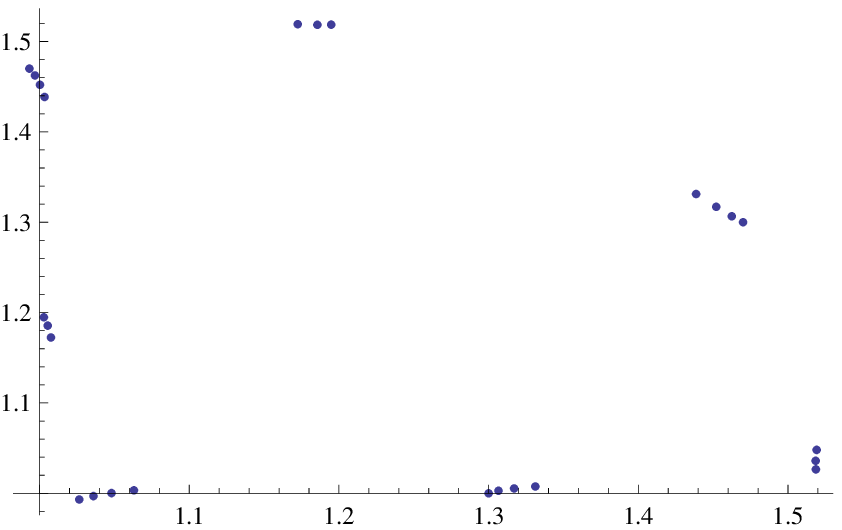}
	\caption{ The trajectory for the case $k=2$, $h=1$, $\tau=4.5$, $x_0=1$, $y_0=1.2$, $x_1=1.3$, $n=25$.}\label{fig-8}
\end{figure}
\begin{figure}[H]
	\includegraphics[width=8cm]{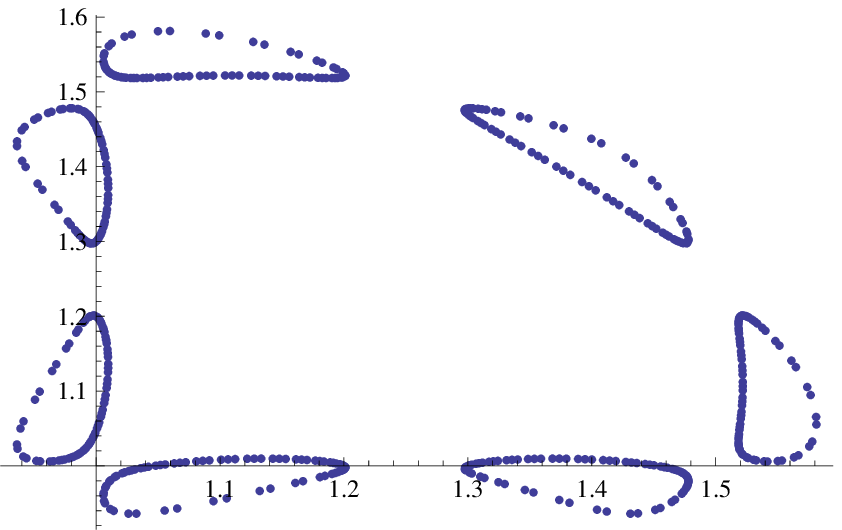}
	\caption{ The trajectory for the case $k=2$, $h=1$, $\tau=4.5$, $x_0=1$, $y_0=1.2$, $x_1=1.3$, $n=500$.}\label{fig-7}
\end{figure}

\begin{figure}[H]
	\includegraphics[width=8cm]{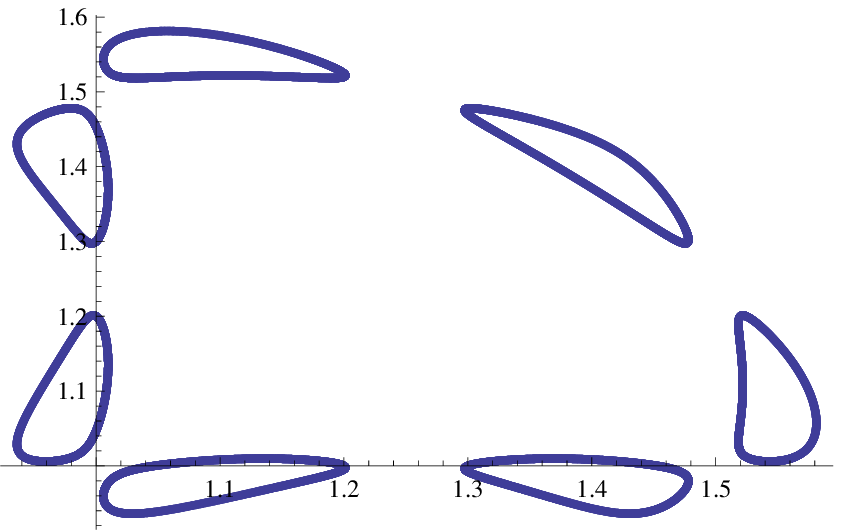}
	\caption{ The trajectory for the case $k=2$, $h=1$, $\tau=4.5$, $x_0=1$, $y_0=1.2$, $x_1=1.3$, $n=10000$.}\label{fig-6}
\end{figure}
\begin{figure}[H]
	\includegraphics[width=8cm]{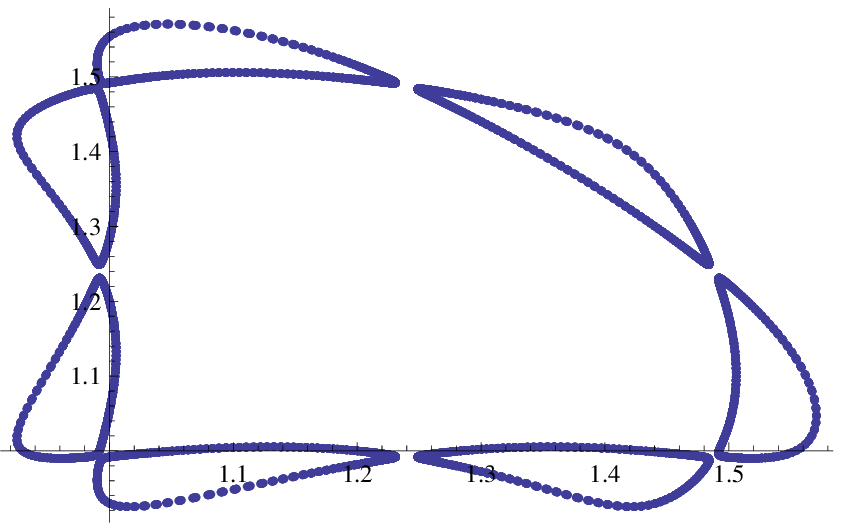}
	\caption{ The trajectory for the case $k=2$, $h=1$, $\tau=4.5$, $x_0=1$, $y_0=1.2$, $x_1=1.2838$, $n=10000$.}\label{fig-9}
\end{figure}
\begin{figure}[H]
	\includegraphics[width=8cm]{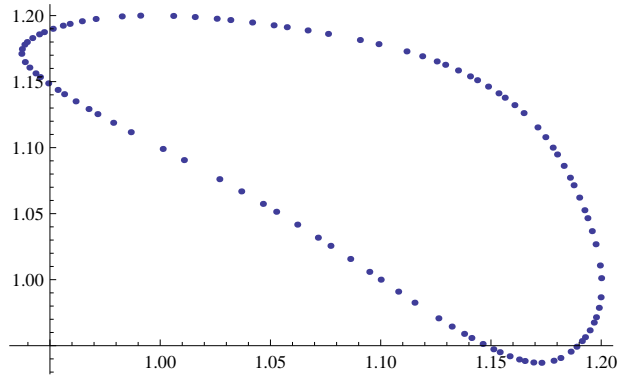}
	\caption{ The trajectory for the case $k=2$, $h=1$, $\tau=5.5$, $x_0=1$, $y_0=1.2$, $x_1=1.1$, $n=100$.}\label{fig-10}
\end{figure}

\subsection{Case $h\ne 1$}
In the case $h\ne 1$ one also can find such trajectories. See Fig. \ref{fig-12}.
\begin{figure}[H]
	\includegraphics[width=8cm]{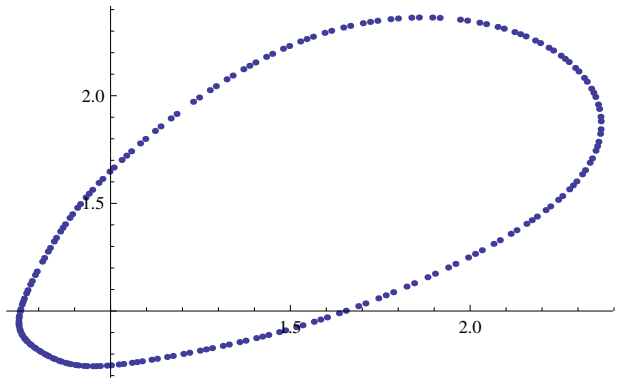}
	\caption{ The trajectory for the case $k=2$, $h=0.5$, $\tau=3$, $x_0=1$, $y_0=1.2$, $x_1=0.6$, $n=200$.}\label{fig-12}
\end{figure}
The following figure shows that there are initial values when the sequence after finitely many steps becomes negative (see Fig. \ref{fig-13}). 
\begin{figure}[H]
	\includegraphics[width=9cm]{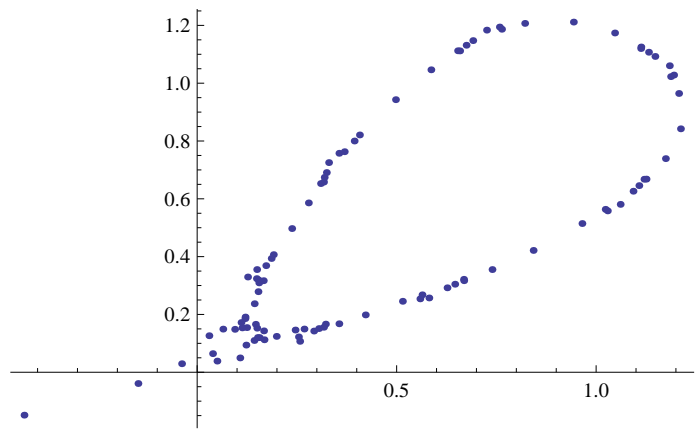}
	\caption{ The trajectory for the case $k=2$, $h=1.05$, $\tau=3$, $x_0=1$, $y_0=1.2$, $x_1=0.6$, $n=95$. This trajectory starting from step 93 goes to negative values. }\label{fig-13}
\end{figure}

\subsection{Gibbs measures of the positive trajectories (\ref{tx}).}
From numerical results of the previous subsections it follows that there are initial points having positive trajectory. Each positive trajectory belongs to a  closed curve which is invariant wish respect to the operator $F$.  In \cite{HKLR}, \cite{HR}, \cite{HKb}, \cite{Hphd} and \cite{Rjsp} many periodic trajectories are found. Trajectories shown in the previous subsections are not periodic,  moreover they do not define some normalisable boundary law. 
Therefore, each such trajectory defines a non-probability Gibbs measure.
  
\section*{ Acknowledgements}
I thank Institut des Hautes \'Etudes Scientifiques (IHES), Bures-sur-Yvette, France for support of his visit to IHES. 
The work was partially supported by a grant from the IMU-CDC.
This research is related to the fundamental project (number: F-FA-2021-425)  of The Ministry of Innovative Development of the Republic of Uzbekistan.

I thank professors I.Shparlinski and John Roberts for helpful discussions and providing related references.

\end{document}